\documentclass[12pt,a4paper]{amsart}

\usepackage[utf8]{inputenc}
\usepackage{graphicx}
\graphicspath{{Figures/}}
\usepackage[all]{xy}
\usepackage{listings}
\usepackage{multicol}

\newcommand{\C}{\mathbb{C}}
\newcommand{\M}{\mathcal{M}}

\newcommand{\Span}{\mathbf{Span}}
\newcommand{\RG}{\mathbf{RG}}
\newcommand{\Grpd}{\mathbf{Grpd}}
\newcommand{\Cat}{\mathbf{Cat}}
\newcommand{\MG}{\mathbf{MG}}
\newcommand{\UMG}{\mathbf{UMG}}
\newcommand{\PreGrpd}{\mathbf{PreGrpd}}
\newcommand{\Kite}{\mathbf{DiKite}}
\newcommand{\MKite}{\mathbf{MKite}}

\newcommand{\cod}{\mathrm{cod}}
\newcommand{\dom}{\mathrm{dom}}
\newcommand{\midop}{\mathrm{mid}}

\newtheorem{definition}{Definition}
\newtheorem{theorem}{Theorem}
\newtheorem{proposition}{Proposition}
\newtheorem{lemma}{Lemma}
\newtheorem{remark}{Remark}

\begin{document}

\date{PAB --- \today.}

\title[On Naturally and Weakly Mal'tsev Categories]{On Naturally and Weakly Mal'tsev Categories}

\author[N. Martins-Ferreira]{Nelson Martins-Ferreira$^{*}$}

\thanks{$(*)$ Polytechnic of Leiria, Portugal}
\address{Department of Mathematics, ESTG, Campus 2\\
Morro do Lena – Alto do Vieiro
Apartado 4163\\
2411-901 Leiria – Portugal}
\address{CDRSP, Rua de Portugal - Zona Industrial\\ 2430-028 Marinha Grande}
\email{martins.ferreira@ipleiria.pt}




\begin{abstract}

We explore a hierarchy of notions in categorical algebra—Mal'tsev categories (where every reflexive relation is symmetric), naturally Mal'tsev categories (where every reflexive graph underlies a unique internal groupoid structure, also known as the Lawvere Condition), and weakly Mal'tsev categories (defined via the joint epimorphicity of local product injections). While varying in generality and historical prominence, these notions can be uniformly described using suitable classes of spans and a new structural condition, called the Kite Condition.

To this end, we reintroduce the notion of a \emph{multiplicative directed kite}, an internal categorical structure that generalizes both internal categories and internal pregroupoids. We also define and study the concept of a \emph{weakly Mal'tsev object}, uncovering new structural insights and examples.

While several of the results we present are known in the context of categories with finite limits, we show that they remain valid—and in fact gain clarity—in the broader setting of categories admitting only pullbacks of split epimorphisms along split epimorphisms.

\end{abstract}

\maketitle

\section{Preamble}

There has long been a debate in the philosophy of mathematics as to whether \emph{good theorems are corollaries of good definitions}, or whether \emph{good definitions are functions of good theorems}. While we do not aim to settle this controversy, in this paper, we search for minimal structural requirements on a category under which key results concerning Mal'tsev, naturally Mal'tsev, and weakly Mal'tsev categories can still be established.

As an illustrative example, we observe that several conditions---typically formulated in the context of categories with finite limits---remain equivalent in any category with local products:
\begin{enumerate}
    \item every reflexive relation is an equivalence relation;
    \item every reflexive relation is transitive;
    \item every reflexive relation is symmetric;
    \item every local product is extremal.
\end{enumerate}

At this level of abstraction, it is widely accepted that any of these equivalent conditions may serve to define a Mal'tsev category (see also Theorem~\ref{them: mainII}). However, a closer analysis reveals that the condition ``every reflexive relation is symmetric'' can be meaningfully formulated even in an arbitrary category---without assuming the existence of any specific limits. This observation is significant given the importance of Mal'tsev categories, which generalize key properties of the category of groups, particularly the behavior of equivalence relations and internal categorical structures.

This line of reasoning aligns with the original definition of a weakly Mal'tsev category and resonates with the spirit of naturally Mal'tsev categories. While each of the three Mal'tsev-type conditions admits several equivalent characterizations under mild limit or colimit assumptions, we adopt definitions that require only the weakest necessary assumptions—thus capturing the most general categorical setting in which each notion can still be meaningfully defined.

As will become apparent, the notion of a \emph{local product}—understood as the pullback of a split epimorphism along another split epimorphism—plays a central role in this framework. For this reason, we begin with a detailed discussion of this construction and introduce some notation that will be used throughout the paper.

\begin{definition}
A \emph{local product} is a diagram of the form
\begin{equation}\label{local product 0}
\vcenter{\xymatrix{
A \ar@<-.5ex>[r]_-{e_1} & E \ar@<-.5ex>[l]_-{p_1} \ar@<.5ex>[r]^-{p_2} & C \ar@<.5ex>[l]^-{e_2}
}}
\end{equation}
satisfying \( p_1 e_1 = 1_A \), \( p_2 e_2 = 1_C \), and for which there exists a second diagram
\begin{equation}\label{local span}
\vcenter{\xymatrix@!0@=4em{
A \ar@<.5ex>[r]^-{f} & B \ar@<.5ex>[l]^-{r} \ar@<-.5ex>[r]_-{s} & C \ar@<-.5ex>[l]_-{g}
}}
\end{equation}
with \( f r = 1_B = g s \), such that the span \( (E, p_1, p_2) \) is the pullback of \( g \) along \( f \). The morphisms
\( e_1 = \langle 1_A, s f \rangle \) and \( e_2 = \langle r g, 1_C \rangle \) are the induced morphisms into the pullback,
forming a split square as shown:
\begin{equation}\label{split square}
\xymatrix@!0@=4em{
E \ar@<.5ex>[r]^-{p_2} \ar@<-.5ex>[d]_-{p_1} & C \ar@<.5ex>[l]^-{e_2} \ar@<-.5ex>[d]_-{g} \\
A \ar@<.5ex>[r]^-{f} \ar@<-.5ex>[u]_-{e_1} & B \ar@<.5ex>[l]^-{r} \ar@<-.5ex>[u]_-{s}.
}
\end{equation}
\end{definition}

In categories where pullbacks of split monomorphisms along split monomorphisms exist—a very mild assumption (see the example below)—local products admit an intrinsic characterization; that is, one which avoids reference to the morphisms \( f \) and \( g \) in the split square above.

\begin{proposition}\label{prop1}
Let $\C$ have pullbacks of split monomorphisms along split monomorphisms. A diagram of the form
\begin{equation}\label{local product (1)}
\vcenter{\xymatrix{
A \ar@<-.5ex>[r]_-{e_1} & E \ar@<-.5ex>[l]_-{p_1} \ar@<.5ex>[r]^-{p_2} & C \ar@<.5ex>[l]^-{e_2}
}}
\end{equation}
is a local product if and only if
\begin{enumerate}
\item \( p_1 e_1 = 1_A \), \( p_2 e_2 = 1_C \);
\item \( e_1 p_1 e_2 p_2 = e_2 p_2 e_1 p_1 \);
\item the pair \( (p_1, p_2) \) is jointly monic;
\item for every \( u\colon Z \to A \) and \( v\colon Z \to C \), if
\begin{align*}
p_1 e_2 p_2 e_1 u &= p_1 e_2 v, \\
p_2 e_1 u &= p_2 e_1 p_1 e_2 v,
\end{align*}
then there exists a unique \( w\colon Z \to E \) such that \( p_1 w = u \) and \( p_2 w = v \).
\end{enumerate}
\end{proposition}
 \begin{proof}
The necessity of all conditions is clear. For sufficiency, observe that the span \( (B, r, s) \) arises as the pullback of the split monomorphism \( e_1 \) along the split monomorphism \( e_2 \). The morphisms \( f \) and \( g \) are then obtained as $
f = \langle p_1 e_2 p_2 e_1,\, p_2 e_1 \rangle$ and
$g = \langle p_1 e_2,\, p_2 e_1 p_1 e_2 \rangle$.
\end{proof}

An example of a category that admits pullbacks of monomorphisms along monomorphisms—but not all pullbacks—is the category whose objects are the natural numbers \( 0, 1, 2, \ldots, N-1 \), for some fixed \( N \), such as \( N = 2^{64} \). A morphism from \( n \) to \( m \) is a pair \( (m, u) \), where \( u = (u_1, \ldots, u_n) \) is a vector of indices of length \( n \), with each \( u_i \leq m \). The domain of \( (m, u) \) is \( n \) (the length of the vector \( u \)), and the codomain is \( m \). Composition is defined by
\[
(m, u) \circ (n, v) = (m, u \circ v),
\]
where \( (u \circ v)(j) = u(v(j)) \), following the usual interpretation of composition for index vectors.

This category admits pullbacks of monomorphisms along arbitrary morphisms, but not all pullbacks. It arises naturally in the modeling of array and indexing operations in programming languages such as \emph{Matlab} and \emph{Octave}. For instance, the built-in function \texttt{ismember} has the signature \texttt{[k, p] = ismember(f, u)}. When \( u \) is a vector with unique entries, the pair \( (\texttt{find}(k), p(k)) \) can be interpreted as the projections of the pullback of \( f \) along \( u \).

\section{Introduction}

This paper continues the work initiated in \cite{MF38}, which itself builds as a corollary of results from \cite{MF17} and \cite{MF7}, both of which trace their origin back to \cite{MF3}, where weakly Mal'tsev categories were first introduced. A recent account of naturally Mal'tsev categories can be found in \cite{Bournetall}.

A comprehensive and historically significant treatment of Mal'tsev and Goursat categories appears in \cite{carboni1993}, which builds on earlier foundational works such as \cite{Fay}, \cite{Mal}, and \cite{Kell}, and is also closely related to \cite{JK}, to mention only the most relevant contributions in relation to our present development.

As remarked in \cite{carboni1993}, it has long been understood that the study of relations can be extended beyond regular categories to any finitely complete category equipped with a proper factorization system \((\mathcal{E}, \mathcal{M})\), by redefining a relation from \( A \) to \( B \) as a subobject of \( A \times B \) belonging to the class \( \mathcal{M} \).

In this paper, we take this generalization one step further: rather than fixing a proper factorization system, 
we consider an arbitrary class $\M$ of spans—possibly not even jointly monic—that serves as the ambient context for \emph{generalized relations}.
Within this framework, many classical results concerning Mal'tsev, naturally Mal'tsev, and weakly Mal'tsev categories are revisited, thereby revealing a broader scope of applicability beyond the usual assumptions of finite limits or regularity. 

More specifically, all of our results hold in any category with local products. Moreover, depending on the choice of the class \( \M \), each of the three Mal'tsev-like notions can be recovered as follows:
\begin{itemize}
    \item taking \( \M \) to be the class of all spans whose legs admit kernel pairs yields the structure of a \emph{naturally Mal'tsev category};
    \item restricting \( \M \) further to the class of jointly monic spans whose legs have kernel pairs gives the notion of a \emph{Mal'tsev category};
    \item and taking \( \M \) to consist of jointly strongly monic spans whose legs have kernel pairs corresponds to the \emph{weakly Mal'tsev} case.
\end{itemize}

Although this level of generality may initially appear abstract, it will be shown to reveal previously unnoticed connections and support the development of novel internal categorical structures.

The paper is organized along the following lines.

First, we review the basic definitions and notations for the internal categorical structures needed in an arbitrary category \( \C \).

From that point onward, we assume that the category \( \C \) has local products; that is, for every diagram of the form
\begin{equation}\label{diag: ABC}
\vcenter{\xymatrix@!0@=4em{
A \ar@<.5ex>[r]^-{f}  & B \ar@<.5ex>[l]^-{r} \ar@<-.5ex>[r]_-{s} & C \ar@<-.5ex>[l]_-{g}
}}
\end{equation}
with \( f r = 1_B = g s \), there exists a canonical diagram
\begin{equation}\label{local product}
\vcenter{\xymatrix{
A \ar@<-.5ex>[r]_-{e_1} & A \times_B C \ar@<-.5ex>[l]_-{\pi_1} \ar@<.5ex>[r]^-{\pi_2} & C \ar@<.5ex>[l]^-{e_2}
}}
\end{equation}
 such that the span \( (A \times_B C, \pi_1, \pi_2) \) is the pullback of \( g \) along \( f \) while \( e_1 = \langle 1_A, s f \rangle \) and \( e_2 = \langle r g, 1_C \rangle \) are the induced morphisms into the pullback.

We then prove the main result that characterizes naturally Mal'tsev categories in the context of local products.

Next, we turn our attention to weakly Mal'tsev categories by introducing the notion of a \emph{weakly Mal'tsev object} and studying some of its consequences, particularly in yielding new examples and characterizations.

Finally, we revisit the unifying result that brings together the three notions—Mal'tsev, naturally Mal'tsev, and weakly Mal'tsev—within the broader setting of categories with local products.

\section{Review of internal categorical structures in arbitrary categories}

The purpose of this section is to establish the notation used throughout the remainder of the paper. At the same time, we aim to propose a standard formulation of the relevant definitions, in the spirit of~\cite{MF47}. An early work that highlights the importance of defining internal categorical structures in arbitrary categories—albeit with a focus on varieties of algebras—is the paper~\cite{JP}.

Throughout the remainder of this section, $\C$ will denote an arbitrary category.

\subsection{Internal Reflexive Graphs}

An internal reflexive graph is a diagram in $\C$ of the form 
\begin{equation}\label{diag: reflexive graph}\xymatrix{C_1 \ar@<1ex>[r]^{d} \ar@<-1ex>[r]_{c} & C_0 \ar[l]|{e} }
\end{equation}
 in which the condition $de=1_{C_0}=ce$ holds true. It can be represented as a five-tuple $(C_1,C_0,d,e,c)$.
A morphism between reflexive graphs, say form $(C_1,C_0,d,e,c)$ to $(C'_1,C'_0,d',e',c')$, is a pair of morphisms $(f_1,f_0)$, displayed as
\begin{equation}\label{diag: morphims of reflexive graphs}\xymatrix{C_1 \ar@<1ex>[r]^{d} \ar@<-1ex>[r]_{c}\ar[d]_{f_1} & C_0 \ar[l]|{e}\ar[d]^{f_0} \\
C'_1 \ar@<1ex>[r]^{d'} \ar@<-1ex>[r]_{c'} & C'_0 \ar[l]|{e'} }
\end{equation}
 and such that $d'f_1=f_0d$, $c'f_1=f_0c$ and $f_1e=e'f_0$.
The category of reflexive graphs internal to $\C$ is denoted $\RG(\C)$.

\subsection{Internal Spans}
An internal span is a diagram in $\C$ of the form 
\begin{equation}\label{diag: span}
\xymatrix{
& D \ar[ld]_{d} \ar[rd]^{c} \\
D_0 && D_1
}
\end{equation}
with no additional conditions. We denote such a span by $(D,d,c)$. The objects $D_0$ and $D_1$ are intentionally omitted, as they are seldom referred to explicitly.
 The category of spans in $\C$ is denoted $\Span(\C)$. There is a canonical functor $\RG(\C) \to \Span(\C)$ that assigns to each reflexive graph $(C_1, C_0, d, e, c)$ the span $(C_1, d, c)$.

  Any class $\M$ of spans in $\C$ can be seen as a full subcategory $\M\to\Span(\C)$. For the sake of consistency we will write $\Span(\C,\M)$ to denote the full subcategory of $\Span(\C)$ determined by the spans in the class $\M$. Similarly we obtain $\RG(\C,\M)$ as the full subcategory of $\RG(\C)$ whose span part is in $\M$, in other words, it can be seen as a pullback of categories 
\[\xymatrix{\RG(\C,\M)\ar[r]\ar[d]&\RG(\C)\ar[d]\\ \Span(\C,\M) \ar[r]&\Span(\C).}\]

\begin{remark}
We will be particularly interested in spans $(D,d,c)$ for which the kernel pairs of both $d$ and $c$ exist. Accordingly, throughout the remainder of this paper, the category $\Span(\C)$ will denote the category whose objects are precisely such spans. Rather than requiring the ambient category $\C$ to admit all kernel pairs, we restrict our attention to those spans whose legs admit kernel pairs. (For further details, see the subsection on the kernel pair construction below.)
\end{remark}

\subsection{Internal Multiplicative Graphs} \label{subsec: multiplicative graph}

In~\cite{MF38} (see also~\cite{JP} for an earlier appearance of this notion), the concept referred to as a \emph{multiplicative graph} corresponds to what we now call a \emph{unital multiplicative graph} (see below).

In this work
we denote by $\MG(\C)$ the category of multiplicative graphs internal to $\C$. Its objects are diagrams in $\C$ of the form
\begin{equation}\label{diag: multiplicative graph}
\xymatrix{
C_2 \ar@<2.5ex>[r]^{\pi_2} \ar@<-2.5ex>[r]_{\pi_1} \ar[r]|{m} & C_1 \ar@<1.5ex>[l]|{e_1} \ar@<-1.5ex>[l]|{e_2} \ar@<1ex>[r]^{d} \ar@<-1ex>[r]_{c} & C_0 \ar[l]|{e}
}
\end{equation}
in which $(C_1, C_0, d, e, c)$ is a reflexive graph, and the following conditions hold:
\begin{align}
d \circ m &= d \circ \pi_2 \\
c \circ m &= c \circ \pi_1,
\end{align}
the square 
\begin{equation}
\xymatrix{
C_2 \ar[r]^{\pi_2} \ar[d]_{\pi_1} & C_1 \ar[d]^{c} \\
C_1 \ar[r]^{d} & C_0
}
\end{equation}
is a pullback, and the morphisms $e_1$ and $e_2$ are uniquely determined as $e_1 = \langle 1_{C_1}, ed \rangle$ and $e_2 = \langle ec, 1_{C_1} \rangle$.

A multiplicative graph, displayed as in diagram~\eqref{diag: multiplicative graph}, will be referred to as a six-tuple $(C_1, C_0, d, e, c, m)$. The canonical morphisms $\pi_1$, $\pi_2$ arising from the pullback, as well as the induced morphisms $e_1$, $e_2$, are understood to be implicit.


Morphisms in $\MG(\C)$ are triples $(f_2, f_1, f_0)$, where $(f_1, f_0)$ is a morphism of reflexive graphs and $f_2 = f_1 \times_{f_0} f_1$ satisfies the conditions:
\[
f_1 \circ m = m' \circ f_2, \quad f_2 \circ e_2 = e'_2 \circ f_1, \quad f_2 \circ e_1 = e'_1 \circ f_1.
\]
When convenient, we write $f \colon C \to C'$ for such a morphism, where $f = (f_2, f_1, f_0)$, and
\[
C = (C_1, C_0, d, e, c, m), \quad C' = (C_1', C_0', d', e', c', m').
\]

There is a canonical forgetful functor from the category of multiplicative graphs $\MG(\C)$ to the category of reflexive graphs $\RG(\C)$, denoted by $F$. Analogously to the construction in the previous subsection, a class $\M$ of spans that defines the subcategory $\RG(\C, \M)$ also determines the subcategory $\MG(\C, \M)$, fitting into a pullback square:
\[
\xymatrix{
\MG(\C, \M) \ar[r]^{F^{\M}} \ar[d] & \RG(\C, \M) \ar[d] \\
\MG(\C) \ar[r]^{F} & \RG(\C).
}
\]
The functor $F^{\M}$ is simply the restriction of $F$ to those multiplicative graphs whose span components lie in $\M$.


Later, we will consider situations in which this forgetful functor admits a section—that is, every reflexive graph in \( \RG(\C, \M) \) is canonically equipped with a compatible multiplication. The notion of a \emph{unital multiplicative graph} (recalled below) ensures that such a multiplication is uniquely determined, provided that suitable conditions are imposed on the span \( (D, d, c) \) and on the induced pair of morphisms \( (e_1, e_2) \) into the pullback.

For example, when $\M$ is the class of all relations, this amounts to the assertion that every reflexive relation is transitive.

\subsection{Internal Unital Multiplicative Graphs}

An internal unital multiplicative graph is an internal multiplicative graph
\begin{equation}\label{diag: unital multiplicative graph}
\xymatrix{
C_2 \ar@<2.5ex>[r]^{\pi_2} \ar@<-2.5ex>[r]_{\pi_1} \ar[r]|{m} & C_1 \ar@<1.5ex>[l]|{e_1} \ar@<-1.5ex>[l]|{e_2} \ar@<1ex>[r]^{d} \ar@<-1ex>[r]_{c} & C_0 \ar[l]|{e}
}
\end{equation}
satisfying the unitality conditions:
\begin{align}
m \circ e_1 &= 1_{C_1}, \\
m \circ e_2 &= 1_{C_1}.
\end{align}

The category of internal unital multiplicative graphs (which in \cite{MF38} was called the category of multiplicative graphs!) is now denoted by $\UMG(\C)$. There are canonical forgetful functors $\UMG(\C) \to \MG(\C)$ and $\UMG(\C) \to \RG(\C)$. We also consider the relative category $\UMG(\C, \M)$.

When the span $(C_1, d, c)$ is jointly monic --- so that the reflexive graph $(C_1, C_0, d, e, c)$ corresponds to a reflexive relation --- the morphism $m$ is uniquely determined, reflecting the fact that the reflexive relation is transitive. In this case, the unitality condition is redundant. However, in the more general situations we will encounter, the multiplication $m$ becomes uniquely determined only once the unitality condition is imposed.

\subsection{The Kernel Pair Construction}\label{subsec: kernel pair construction}

On several occasions, we will require a method for transforming a span into a reflexive graph. This will be achieved using a general construction, referred to as the \emph{kernel pair construction}. The applicability of this construction relies on two key assumptions: that the given spans admit kernel pairs, and that the category $\C$ has local products. These assumptions are essential.

Let $(D, d, c)$ be a span. The kernel pair construction proceeds by forming the kernel pairs of the morphisms $d$ and $c$, and combining them with the pullbacks of their projections and the induced morphisms, as illustrated below:
\[
\xymatrix@!0@=4em{
D(d,c) \ar@<.5ex>[r]^-{p_2} \ar@<-.5ex>[d]_-{p_1}
& D(c) \ar@<.5ex>[l]^-{e_2} \ar@<-.5ex>[d]_-{c_1} \ar@<0ex>[r]^{c_2}
& D \ar[d]_{c}
 \\
D(d)
 \ar@<.5ex>[r]^-{d_2} \ar@<-.5ex>[u]_-{e_1} \ar[d]_{d_1}
& D
 \ar@<.5ex>[l]^-{\Delta} \ar@<-.5ex>[u]_-{\Delta} \ar[r]^{c} \ar[d]_{d}
& D_1 \\
D \ar[r]^{d}
& D_0.}
\]

When $\C$ is the category of sets and maps, we may see an element in $D$ as a sort of an arrow in which its domain and codomain are from different sets, so that $x\in D$ may be  displayed as
\[\xymatrix{D_0\ni d(x)\ar[r]^{x} & c(x)\in D_1,}\]
with this interpretation the elements in $D(d)$ are the pairs $(x,y)$, where $x,y\in D$ are such that $d(x)=d(y)$ and they may be pictured as
\[\xymatrix{c(x) & d(x)=d(y) \ar[l]_-{x} \ar[r]^-{y} & c(y)}\] or in a simpler form as 
\[\xymatrix{\cdot & \cdot \ar[l]_-{x} \ar[r]^-{y} & \cdot}\]
similarly, a pair $(y,z)\in D(c)$ is pictured as
\[\xymatrix{\cdot  \ar[r]^-{y} & \cdot & \cdot \ar[l]_-{z} }\]
in the same way, the elements in $D(d,c)$ are the triples $(x,y,z)$ such that $d(x)=d(y)$ and $c(y)=c(z)$, which may be pictured as
\[\xymatrix{\cdot & \cdot \ar[l]_{x} \ar[r]^-{y} & \cdot & \cdot \ar[l]_-{z} }\]

In other words, when $\C$ is the category of sets and maps we have:
\begin{eqnarray*}
d_1(x,y)=x\\
d_2(x,y)=y\\
c_1(y,z)=y\\
c_2(y,z)=z\\
\Delta(y)=(y,y)\\
p_1(x,y,z)=(x,y)\\
p_2(x,y,z)=(y,z)\\
e_1(x,y)=(x,y,y)\\
e_2(y,z)=(y,y,z).
\end{eqnarray*}

The kernel pair construction gives rise to a functor
\[
K \colon \Span(\C) \to \RG(\C),
\]
defined on objects by
\[
K(D, d, c) = \big(D(d, c),\, D,\, d_1 \circ p_1,\, e_1 \circ \Delta,\, c_2 \circ p_2\big).
\]
For later reference, we introduce the following notation:
\[
\dom := d_1 \circ p_1, \qquad \midop := d_2 \circ p_1 = c_1 \circ p_2, \qquad \cod := c_2 \circ p_2.
\]
With this, for any three morphisms $x, y, z \colon X \to D$ such that $dx = dy$ and $cy = cz$, we have
\[
\dom\langle x, y, z \rangle = x, \qquad \midop\langle x, y, z \rangle = y, \qquad \cod\langle x, y, z \rangle = z.
\]

We will also adopt the suggestive (and slightly abusive) notation
\[
\Delta(x) = \langle x, x, x \rangle,
\]
so that $\Delta \colon D \to D(d, c)$ denotes the diagonal morphism mapping an element to the triple with all entries equal.

With this notation, the kernel pair construction simplifies to
\[
K(D, d, c) = (D(d, c),\,D,\, \dom,\, \Delta,\, \cod).
\]

In some contexts, it will be convenient to consider the kernel pair construction where the roles of $d$ and $c$ are interchanged. This produces a reflexive graph
\[
K(D, c, d) = (D(c, d),\,D,\, \dom,\, \Delta,\, \cod),
\]
which is interpreted informally by the diagram
\[
\xymatrix{\cdot \ar[r]^-{x} & \cdot & \cdot \ar[l]_-{y} \ar[r]^-{z} & \cdot}
\]
where $(x, y, z) \in D(c, d)$. In this case, the structural maps act as
\[
\dom\langle x, y, z \rangle = z, \qquad \midop\langle x, y, z \rangle = y, \qquad \cod\langle x, y, z \rangle = x.
\]

An alternative way to define the kernel pair construction—assuming the existence of all pullbacks and binary products—is via the pullback
\[
\xymatrix{
D(d, c) \ar[r] \ar[d] & D \ar[d]^{\langle d, c \rangle} \\
D \times D \ar[r]^{d \times c} & D_0 \times D_1.
}
\]

Given a class of spans $\M$, we say that $\M$ is \emph{pullback stable} if, for every span $(D, d, c)$ in $\M$ and every pair of morphisms $u \colon U \to D_0$ and $v \colon V \to D_1$, the span $(B, d', c')$ obtained by taking pullbacks as illustrated below
\[
\xymatrix{
& & B \ar[dl] \ar@<-.5ex>@/_/[lldd]_{d'} \ar[rd] \ar@<.5ex>@/^/[rrdd]^{c'} \\
& A \ar[ld] \ar[rd] && C \ar[ld] \ar[rd] \\
U \ar[rd]_{u} & & D \ar[ld]^{d} \ar[rd]_{c} && V \ar[ld]^{v} \\
& D_0 && D_1
}
\]
also belongs to $\M$.

This aligns with the standard notion of pullback stability: when pullbacks and binary products exist, the span $(B, d', c')$ can be obtained via the pullback
\[
\xymatrix{
B \ar[r] \ar[d]_{\langle d', c' \rangle} & D \ar[d]^{\langle d, c \rangle} \\
U \times V \ar[r]^{u \times v} & D_0 \times D_1,
}
\]
so that if $(D, d, c) \in \M$, then $(B, d', c') \in \M$ as well.

The requirement that the class $\M$ be pullback stable is explicitly used in the proof of the main theorem in~\cite{MF38}. However, we will show that it is sufficient for $\M$ to satisfy the following two weaker conditions, assuming that the ambient category admits local products:
\begin{enumerate}
\item If $(D, d, c) \in \M$, then the kernel pairs of $d$ and $c$ exist, and the associated span arising from the kernel pair construction, $(D(c,d), \dom, \cod)$, also lies in $\M$;
\item The span component of every local product belongs to $\M$.
\end{enumerate}

\subsection{Pregroupoids}
 
 A pregroupoid internal to $\C$ is a span 
\[\xymatrix{& D \ar[ld]_{d} \ar[rd]^{c}\\D0 && D1}\]
together with a pregroupoid structure (note that we are only considering those spans for which the kernel pair construction exists). A pregroupoid structure is  a morphism $p\colon{D({d,c})\to D}$, such that 
\begin{gather}
p e_1 =d_1\quad\text{and}\quad p e_2 =c_2,\label{Mal'tsev-conditions}\\
dp =d c_2 p_2\quad\text{and}\quad cp = c d_1 p_1.\label{Domain-and-Codomain}
\end{gather}

The object $D(d,c)$ is obtained together with the maps $$d_1,d_2,c_1,c_2,p_1,p_2,e_1,e_2$$ by means of the kernel pair construction. As explained in the previous subsection. In set-theoretical terms, the object $D({d,c})$ consists on those triples $(x,y,z)$ of arrows in $D$ for which $d(x)=d(y)$ and $c(y)=c(z)$, so that the conditions $(\ref{Mal'tsev-conditions})$ can be translated as
\[p(x,y,y)=x,\quad p(y,y,z)=z\] while the conditions $(\ref{Domain-and-Codomain})$ mean
\[dp(x,y,z)=d(z), \quad cp(x,y,z)=c(x). \]

In this way we may also form the category of pregroupoids with its span part drawn from a class $\M$. It will be denoted as $\PreGrpd(\C,\M)$.

A pregroupoid is associative if for every $x,y,z\in D$, as above, together with $u,v\in D$ with $d(z)=d(u)$ and $c(u)=c(v)$, we have
\begin{equation}\label{associative for a pregroupoid}
p(p(x,y,z),u,v)=p(x,y,p(z,u,v)).
\end{equation}

An associative pregroupoid is equivalent to a Kock pregroupoid \cite{K1,K2}, while the concept of a pseudogroupoid is studied in \cite{JP2} in connection with commutator theory. A recent account of these notions can be found in \cite{Bournetall}, where the notion of an autonomous pregroupoid is also considered.

In future work, we intend to investigate the relationships between the following conditions, interpreted within the framework of commutator theory as developed in \cite{JP2}, under the standing assumptions of local products and a relative class of spans admitting the kernel pair construction.

\begin{enumerate}
\item The forgetful functor from the category of pseudogroupoids to the category of spans is an isomorphism;
\item The forgetful functor from the category of autonomous pregroupoids to the category of spans is an isomorphism;
\item The forgetful functor from the category of associative pregroupoids to the category of spans is an isomorphism;
\item The forgetful functor from the category of pregroupoids to the category of spans is an isomorphism;
\item The forgetful functor from the category of pseudogroupoids to the category of spans has a section;
\item The forgetful functor from the category of autonomous pregroupoids to the category of spans has a section;
\item The forgetful functor from the category of associative pregroupoids to the category of spans has a section;
\item The forgetful functor from the category of pregroupoids to the category of spans has a section.
\end{enumerate}

Part of this study (not involving pseudogroupoids) was sketched in~\cite{MF17} in the context of local products, kernel pairs, and coequalizers. It is now understood that the assumption of existing kernel pairs can be relaxed by restricting the analysis to spans whose legs admit kernel pairs. We further hope to eliminate the need for coequalizers entirely.

\subsection{Internal categories and internal groupoids}

An internal category is a unital multiplicative graph in which the multiplication is associative. The category of internal categories to $\C$ is denoted $\Cat(\C)$. A groupoid is a an internal category with inverses. It can be seen as an associative unital multiplicative graph in which the square 
\begin{equation}
\xymatrix{C_2 \ar[r]^{\pi_2}\ar[d]_{\pi_1} & C_1 \ar[d]^{d}\\
 C_1\ar[r]^{d} & C_0}
\end{equation}
is a pullback (see e.g. \cite{MF3}). 
The category of internal groupoids internal to $\C$ is denoted $\Grpd(\C)$. 



In a similar manner as before we define the categories $\Cat(\C,\M)$ and $\Grpd(\C,\M)$ of internal categories and internal groupoids in $\C$ with respect to a given class $\M$.

\subsection{Multiplicative directed kites}

A kite internal to $\C$ is a diagram of the form
\begin{equation}\label{diag: kite}
\vcenter{\xymatrix@!0@=4em{A \ar@<.5ex>[r]^-{f} \ar[rd]_-{\alpha} & B
\ar@<.5ex>[l]^-{r}
\ar@<-.5ex>[r]_-{s}
\ar[d]^-{\beta} & C \ar@<-.5ex>[l]_-{g} \ar[ld]^-{\gamma}\\
& D}}
\end{equation}
such that $fr=1_{B}=gs$, $\alpha r=\beta=\gamma s$.

A directed kite is a kite together with a span $(D,d,c)$ such that $d\alpha=d\beta f$, $c\beta g=c\gamma$.

 Once again, if the span part of a kite is required to be in $\M$ then it is an object in the  category $\Kite(\C,\M)$, where the morphisms are the natural transformations between such diagrams.

Under the presence of local products, each diagram such as $(\ref{diag: kite})$ induces a diagram
\begin{equation}\label{diag: kite with local product}
\vcenter{\xymatrix@!0@=3em{ & C \ar@<.5ex>[ld]^-{e_2} \ar@<-.5ex>[rd]_-{g}
\ar@/^/[rrrd]^-{\gamma} \\
A\times_{B}C \ar@<.5ex>[ru]^-{\pi_2}
\ar@<-.5ex>[rd]_-{\pi_1} && B \ar@<.5ex>[ld]^-{r} \ar@<-.5ex>[lu]_-{s}
 \ar[rr]|-{\beta} && D\\
& A \ar@<.5ex>[ru]^-{f} \ar@<-.5ex>[lu]_-{e_1} \ar@/_/[urrr]_-{\alpha}}}\end{equation}
in which the double diamond  is constructed as a local product. It can be seen as a double split epimorphism as well. Recall that the morphisms $e_1, e_2$ are given by $e_1=\langle 1_A, sf\rangle$ and $e_2=\langle rg, 1_C\rangle$.

A multiplication on a directed kite is a morphism $m\colon{A\times_{B} C\to D}$ such that $dm=d\gamma\pi_2$, $cm=c\alpha\pi_1$, $me_1=\alpha$ and $me_2=\gamma$.

Once again, we will be interested in considering the forgetful functor from the category of multiplicative directed kites into the category of directed kites, with direction (that is the span part) drawn from a given class $\M$.

\subsection{Particular cases}

We have now defined the  categories and the respective canonical functors that we will consider between them.

\[
\xymatrix{
\Grpd(\C,\M)\ar[d]\ar[r]^{F_4} & \RG(\C,\M)\ar@{=}[d]\\
\Cat(\C,\M)\ar[d]\ar[r]^{F_3} & \RG(\C,\M)\ar@{=}[d]\\
\UMG(\C,\M)\ar[r]^{F_2} & \RG(\C,\M)\ar@<-.5ex>[d]\\
\PreGrpd(\C,\M)\ar[u]\ar[d]\ar[r]^{F_1} & \Span(\C,\M)\ar@<-.5ex>[d]\ar@<-.5ex>[u] & \M\ar[l]_-{\cong}\\
\MKite(\C,\M)\ar[r]^{F_0} & \Kite(\C,\M)\ar@<-.5ex>[u]_{}}
\]
Let us also present a selection of examples and special cases of directed kites, along with an analysis of the existence and uniqueness of their associated multiplications.

\begin{enumerate}
\item If $(C_1,C_0,d,e,c)$ is a reflexive graph then the following diagram is a directed kite
\begin{equation}\label{diag: kite1}
\vcenter{\xymatrix{C_1 \ar@<.5ex>[r]^-{d} \ar@{=}[rd]_-{} & C_0
\ar@<.5ex>[l]^-{e}
\ar@<-.5ex>[r]_-{e}
\ar[d]^-{e} & C_1 \ar@<-.5ex>[l]_-{c} \ar@{=}[ld]^-{}\\
& C_1\ar[dl]_{d}\ar[rd]^{c}\\\cdot&&\cdot}}
\end{equation}

The collection of multiplicative structures on this directed kite is in one-to-one correspondence with the collection of unital multiplicative structures on the given reflexive graph.

\item If $(C_1,C_0,d,e,c,m)$ is a unital multiplicative graph then the following diagram is a directed kite.
\begin{equation}\label{diag: kite2}
\vcenter{\xymatrix{C_2 \ar@<.5ex>[r]^-{\pi_2} \ar[rd]_-{m} & C_1
\ar@<.5ex>[l]^-{e_2}
\ar@<-.5ex>[r]_-{e_1}
\ar@{=}[d]^-{} & C_1 \ar@<-.5ex>[l]_-{\pi_1} \ar[ld]^-{m}\\
& C_1\ar[dl]_{d}\ar[rd]^{c}\\\cdot&&\cdot}}
\end{equation}

This directed kite has a unique multiplication if and only if the unital multiplicative graph is associative, that is,  an internal category.

\item If $(C_1,C_0,d,e,c,m)$ is  an internal category then the following diagram is a directed kite
\begin{equation}\label{diag: kite3}
\vcenter{\xymatrix{C_2 \ar@<.5ex>[r]^-{m} \ar[rd]_-{\pi_2} & C_1
\ar@<.5ex>[l]^-{e_2}
\ar@<-.5ex>[r]_-{e_1}
\ar@{=}[d]^-{} & C_1 \ar@<-.5ex>[l]_-{m} \ar[ld]^-{\pi_1}\\
& C_1\ar[dl]_{d}\ar[rd]^{c}\\\cdot&&\cdot}}
\end{equation}

This directed kite has a unique multiplication if and only if the internal category is an internal groupoid (see \cite{MF3}). 

\item If $(f_1,f_0)\colon{(C_1,C_0,d,e,c)\to (C'_1,C'_0,d',e',c')}$ is a morphism of reflexive graphs then the following diagram is a directed kite
\begin{equation}\label{diag: kite4}
\vcenter{\xymatrix{C_1 \ar@<.5ex>[r]^-{d} \ar@{->}[rd]_-{f_1} & C_0
\ar@<.5ex>[l]^-{e}
\ar@<-.5ex>[r]_-{e}
\ar[d]^-{e'f_0} & C_1 \ar@<-.5ex>[l]_-{c} \ar@{->}[ld]^-{f_1}\\
& C'_1\ar[dl]_{d'}\ar[rd]^{c'}\\\cdot&&\cdot}}
\end{equation}
If this directed kite has a unique multiplication and if the reflexive graphs $C$ and $C'$ are unital multiplicative graphs, with  multiplications $m\colon{C_2\to C_1}$ and $m'\colon{C'_2\to C'_1}$, then $f_1 m=m'f_2$. The notation is borrowed from subsection \ref{subsec: multiplicative graph}.

\item If $(D,d,c)$ is a span then the kernel pair construction gives a directed kite as follows
\begin{equation}\label{diag: kite5}
\vcenter{\xymatrix{D(d) \ar@<.5ex>[r]^-{d_2} \ar@{->}[rd]_-{d_1} & D
\ar@<.5ex>[l]^-{\Delta}
\ar@<-.5ex>[r]_-{\Delta}
\ar@{=}[d]^-{} & D(c) \ar@<-.5ex>[l]_-{c_1} \ar@{->}[ld]^-{c_2}\\
& D\ar[dl]_{d}\ar[rd]^{c}\\\cdot&&\cdot}}
\end{equation}

This yields a reflection between the category of directed kites  and the category of spans

\[\xymatrix{\Kite \ar@<.5ex>[r]  &\Span\ar@<.5ex>[l]}\]

A directed kite goes to its direction span, a span goes to the directed kite displayed above. Moreover, the span $(D,d,c)$ is a pregroupoid if and only if its associated directed kite is multiplicative.

\item Consider a morphism between two directed kites as illustrated:
\begin{equation}\label{diag: kite6}
\vcenter{\xymatrix@!0@=4em{
A \ar@<.5ex>[r]^-{f} \ar[rd]_-{\alpha} & B
\ar@<.5ex>[l]^-{r}
\ar@<-.5ex>[r]_-{s}
\ar[d]^-{\beta} & C \ar@<-.5ex>[l]_-{g} \ar[ld]^-{\gamma}
&
A' \ar@<.5ex>[r]^-{f'} \ar[rd]_-{\alpha'} & B'
\ar@<.5ex>[l]^-{r'}
\ar@<-.5ex>[r]_-{s'}
\ar[d]^-{\beta'} & C' \ar@<-.5ex>[l]_-{g'} \ar[ld]^-{\gamma'}
\\
& D\ar[rrr]^{h_D}\ar[dl]_{d}\ar[rd]^{c}
&&& D'\ar[dl]_{d'}\ar[rd]^{c'}
\\
D_0&&D_1
& D'_0&&D'_1
}}
\end{equation}
The morphism $h_D$ is accompanied by additional morphisms
\(
h_0\colon D_0 \to D'_0,\quad
h_1\colon D_1 \to D'_1,\quad
h_A\colon A \to A',\quad
h_B\colon B \to B',\quad
h_C\colon C \to C'
\) 
which together satisfy the necessary compatibility conditions, making the full diagram commute. These data define what we refer to as a morphism of directed kites.

This morphism induces the following diagram:
\begin{equation}\label{diag: kite7}
\vcenter{\xymatrix@!0@=4em{
A \ar@<.5ex>[r]^-{f} \ar[rd]_-{h_D\alpha} & B
\ar@<.5ex>[l]^-{r}
\ar@<-.5ex>[r]_-{s}
\ar[d]^-{h_D\beta} & C \ar@<-.5ex>[l]_-{g} \ar[ld]^-{h_D\gamma}
\\
& D'\ar[dl]_{d'}\ar[rd]^{c'}
\\
D'_0&&D'_1
}}
\end{equation}
which is itself a directed kite.

Moreover, suppose this directed kite admits a unique multiplication, and that both of the original directed kites are multiplicative, with respective multiplication morphisms
\[
m\colon A \times_B C \to D
\quad \text{and} \quad
m'\colon A' \times_{B'} C' \to D'.
\]
Then the following compatibility condition holds
\[
h_D \circ m = m' \circ (h_A \times_{h_B} h_C),
\]
where $h_A \times_{h_B} h_C$ denotes the induced morphism from the pullback $A \times_B C$ to $A' \times_{B'} C'$.

\end{enumerate}

\subsection{The Kite Condition}

The \emph{Kite Condition} is a slight refinement of the notion of a multiplicative directed kite. It will play a central role in the following sections, particularly in the characterization of naturally Mal'tsev categories.

We say that the Kite Condition holds in a category \( \C \) if, for every diagram of the form:
\begin{equation}\label{diag: the kite condition}
\vcenter{\xymatrix@!0@=4em{& E \ar[dd]^-{\beta}
\ar@<.5ex>[dl]^-{p_1}
\ar@<-.5ex>[dr]_-{p_2}
\\A \ar@<.5ex>[ru]^-{e_1} \ar[rd]_-{\alpha} &  & C \ar@<-.5ex>[lu]_-{e_2} \ar[ld]^-{\gamma}\\
& D\ar[dl]_{c}\ar[rd]^{d}\\D_0&&D_1}}
\end{equation}
with:
\begin{enumerate}
\item $p_1e_1=1_A$, $p_2e_2=1_C$;
\item $(e_1p_1)(e_2p_2)=(e_2p_2)(e_1p_1)$
\item $(p_1,p_2)$ is jointly monoic;
\item $\alpha(p_1e_2p_2)=\beta=\gamma(p_2e_1p_1)$
\item $d\alpha p_1=d\alpha(p_1e_2p_2
)$, $c\gamma p_2=c\gamma(p_2e_1p_1)$
\item $d$ and $c$ have kernel pairs and admit the kernel pair construction
\end{enumerate}
there exists a unique morphism $m\colon{E\to D}$ such that
\begin{align*}
me_1=\alpha\\
me_2=\gamma\\
dm=d\gamma p_2\\
cm=c\alpha p_1.
\end{align*}

Note that conditions (1)--(3) above are the same as conditions (1)--(3) in Proposition \ref{prop1}. 

We are now in position to give a characterization on naturally Mal'tsev categories in any category with local products.

\section{Naturally Mal'tsev categories}

From now on we will assume that our ambient category $\C$ has local products (see the end of the introductory section).

In this setting, a span $(D,d,c)$ for which the kernel pairs of $d$ and $c$ exist, will have the kernel pair construction. This means that the last hypoteses on the kite condition can be relaxed to asking that the kernel pairs of $d$ and $c$ exist.

\begin{lemma}\label{lemma 1} Let $\C$ be a category with local products. The following conditions are equivalent:
\begin{enumerate}
\item The forgetful functor $\UMG(\C)\to\RG(\C)$ is an isomorphism.
\item The forgetful functor $\UMG(\C)\to\RG(\C)$ has a section.
\item The Kite Condition holds in $\C$.
\end{enumerate}
\end{lemma} 
\begin{proof}
Clearly (1) implies (2). Let us prove that (2) implies (3). Take any kite diagram satisfying the prescribed hypoteses. In particular, from the span $(D,d,c)$ we form the reflexive graph $$(D(c,d),D,\dom,\Delta,\cod)$$ and by the section to the forgetful functor $\UMG(\C)\to\RG(\C)$ we find $\mu\colon{D(c,d)\times_D D(c,d)\to D(c,d)}$. Moreover, there exists a morphism $\theta\colon{E\to D(c,d)\times_D D(c,d)}$ defined as $\theta=\langle\langle\alpha p_1,\alpha p_1,\beta\rangle,\langle\beta,\gamma p_2,\gamma p_2\rangle\rangle$. The desired morphism $m\colon{E\to D}$ is obtained as $m=\midop\circ\mu\circ \theta$. The notation $D(c,d)$ as well as the morphisms $\dom,\midop,\cod\colon{D(c,d)\to D}$ and $\Delta:D\to D(c,d)$ are obtained as discussed in the subsection on the kernel pair construction. This proves existence.

In order to prove uniqueness we observe that there exists a morphism $\delta\colon{E\to E(p_1,p_2)\times_E E(p_1,p_2)}$ defined as \begin{align*}
\delta=\langle
\langle e_1p_1,e_1p_1,e_1p_1e_2p_2
\rangle,\langle e_1p_1e_2p_2,e_2p_2,e_2p_2
\rangle\rangle
\end{align*}
and it satisfies $\midop\circ\mu\circ
\delta=1_E$. Indeed, since $(p_1,p_2)$ is jointly monic and we have
\begin{align*}
p_1\circ\midop \circ\mu\circ\delta &=p_1\circ\cod\circ
\mu\circ\delta\\
&=p_1\circ\cod\circ
\pi_1\circ\delta\\
&=p_1\circ\cod\circ\langle e_1p_1,e_1p_1,e_1p_1e_2p_2
\rangle\\
&=p_1e_1p_1=p_1
\end{align*}
and
\begin{align*}
p_2\circ\midop \circ\mu\circ\delta &=p_2\circ\dom\circ
\mu\circ\delta\\
&=p_2\circ\dom\circ
\pi_2\circ\delta\\
&=p_2\circ\dom\circ\langle e_1p_1e_2p_2,e_2p_2,e_2p_2
\rangle\\
&=p_2e_2p_2=p_2
\end{align*}
we may conclude that $\midop\circ\mu\circ
\delta=1_E$.
Now, by the  naturality of $\mu$, as displayed,
\begin{equation}\label{diag: mu is natural}
\vcenter{\xymatrix{
E(p_1,p_2)\times_E E(p_1,p_2) \ar[d]_{m^3\times_m m^3}\ar@<0ex>[r]^-{\mu} & E(p_1,p_2) \ar@<0ex>[r]^-{\midop} \ar@<0ex>[d]^-{m^3} & E \ar@<0ex>[d]^-{m}\\
D(c,d)\times_D D(c,d) \ar[r]^-{\mu} & D(c,d) \ar[r]^-{\midop} & D
}}
\end{equation}
 and observing further that $\theta=(m^3\times_m m^3)\delta$ we obtain

 \begin{align*}
  m &= m1_E=m\circ\midop\circ\mu
  \circ\delta\\
  &=\midop\circ\mu\circ (m^3\times_m m^3)\circ \delta\\
  &=\midop\circ\mu
  \circ\theta. 
 \end{align*}
 This means that $m$ is uniquely determined as $\midop\circ\mu\circ\theta 
$.

It remains to prove that (3) implies (1), which will be shown in Theorem \ref{thm Main}.
\end{proof}

A straightforward observation shows that the unital multiplicative graph needed in the proof of the previous lemma (see the proof of Theorem \ref{thm Main}) is in fact associative, and therefore defines an internal category. This motivates the following result.

\begin{proposition} Let $\C$ be a category with local products. The following conditions are equivalent:
\begin{enumerate}
\item The forgetful functor $\Cat(\C)\to\RG(\C)$ is an isomorphism.
\item The forgetful functor $\Cat(\C)\to\RG(\C)$ has a section.
\item The Kite Condition holds in $\C$.
\end{enumerate}
\end{proposition} 
\begin{proof}
It is clear that (1) implies (2), and that (2) implies the existence of a section for the forgetful functor \( \UMG(\C) \to \RG(\C) \). By the previous lemma, this yields the Kite Condition. 

To complete the proof, it remains to show that the unital multiplicative graph needed in the proof of the previous lemma is associative. It suffices to observe that both morphisms \( 1 \times m \) and \( m \times 1 \) serve as multiplications for the directed kite displayed in~\eqref{diag: kite2}, and therefore must be equal.
\end{proof}

A further observation shows that the internal category obtained in the previous proposition is, in fact, an internal groupoid. This motivates the following result.

\begin{proposition} Let $\C$ be a category with local products. The following conditions are equivalent:
\begin{enumerate}
\item The forgetful functor $\Grpd(\C)\to\RG(\C)$ is an isomorphism.
\item The forgetful functor $\Grpd(\C)\to\RG(\C)$ has a section.
\item The Kite Condition holds in $\C$.
\end{enumerate}
\end{proposition} 
\begin{proof}
It is clear that (1) implies (2), and that (2) implies the existence of a section for the forgetful functor \( \UMG(\C) \to \RG(\C) \). By the previous lemma, this yields the Kite Condition.

To complete the proof, it remains to show that the internal category constructed in the previous proposition is an internal groupoid. This follows from the fact that the directed kite displayed in~\eqref{diag: kite3} admits a unique multiplication, by the kite condition, and thus determines an internal groupoid.

\end{proof}


The condition that the forgetful functor from internal groupoids to reflexive graphs is an isomorphism is known as the \emph{Lawvere Condition}, and it is customary to use it in the definition of a naturally Mal'tsev category. The following characterization of naturally Mal'tsev categories strengthens existing results in the literature, in that it holds in any category with local products, without requiring the existence of finite limits. Moreover, the Kite Condition presented here refines the compatibility condition between spans and split squares considered in~\cite{MF38}.

\begin{theorem}\label{thm Naturally} Let $\C$ be a category with local products. The following conditions are equivalent:
\begin{enumerate}
\item The forgetful functor $\Grpd(\C)\to\RG(\C)$ is an isomorphism.
\item The forgetful functor $\Grpd(\C)\to\RG(\C)$ has a section.
\item The forgetful functor $\Cat(\C)\to\RG(\C)$ is an isomorphism.
\item The forgetful functor $\Cat(\C)\to\RG(\C)$ has a section.
\item The Kite Condition holds in $\C$.
\end{enumerate}
Moreover, when $\C$ has binary products and a terminal object, the above conditions are further equivalent to the existence of a natural transformation
\begin{align*}
p_D\colon{D\times D\times D\to D},
\end{align*}
with $D$ in $\C$, such that for every two morphisms $x,y\colon{Z\to D}$, \begin{align*}
p\langle x,y,y\rangle=x ,\quad p\langle x,x,y\rangle=y.
\end{align*}

\end{theorem}

A classical example of a naturally Mal'tsev category is the category of abelian groups. We also present a lesser-known example: a subcategory of  commutative cancellative medial magmas, as discussed in \cite{MF26}.

Recall that a \emph{ commutative cancellative medial magma} is an algebraic structure $(A, \cdot)$ consisting of a set $A$ equipped with a binary operation $\cdot$ satisfying the following three conditions:
\begin{align*}
x \cdot y &= y \cdot x, \quad &&\text{(commutativity)} \\
(x \cdot y) \cdot (z \cdot w) &= (x \cdot z) \cdot (y \cdot w), \quad &&\text{(mediality)}
\end{align*}
for all $x, y, z, w \in A$, together with the cancellation condition:
\begin{align*}
\forall x, y \in A,\, \big(\exists b \in A,\, x \cdot b = y \cdot b\big) \Rightarrow x = y.
\end{align*}

\begin{proposition}
Let $\C$ be a full subcategory of the category of commutative cancellative medial magmas that is closed under finite limits. Then $\C$ is a naturally Mal'tsev category if and only if, for every object $(D, \cdot)$ in $\C$ and all elements $a, b, c \in D$, the equation
\[
x \cdot b = a \cdot c
\]
admits a (unique) solution $x \in D$.
\end{proposition}
\begin{proof}
Since $\C$ has finite limits, we may use the characterization that $\C$ is naturally Mal'tsev if and only if each object $D$ is equipped with a natural Mal'tsev morphism. 

Suppose such a morphism $p \colon D \times D \times D \to D$ exists. Then, given any $a, b, c \in D$, define $x := p(a, b, c)$. We claim that this $x$ is a solution to the equation $x \cdot b = a \cdot c$. Indeed,
\begin{align*}
x \cdot b &= p(a, b, c) \cdot p(b, b, b) \\
&= p(a \cdot b, b \cdot b, c \cdot b) && \text{(since $p$ is a morphism)} \\
&= p(a \cdot b, b \cdot b, b \cdot c) && \text{(commutativity)} \\
&= p(a, b, b) \cdot p(b, b, c) && \text{(since $p$ is Mal'tsev)} \\
&= a \cdot c.
\end{align*}
 The cancellation property ensures that this solution is unique.

Conversely, suppose that for every $a, b, c \in D$ the equation $x \cdot b = a \cdot c$ admits a unique solution $x \in D$. Then we define
\[
p(a, b, c) := \text{the unique $x$ such that } x \cdot b = a \cdot c.
\]
Clearly, $p(a, b, b) = a$, since $a \cdot b = a \cdot b$, and $p(b, b, c) = c$, since $c \cdot b = b \cdot c$ by commutativity.

Let $f \colon D \to D'$ be a morphism in $\C$. We verify that $p$ is natural. Observe:
\begin{align*}
f(p(a, b, c)) \cdot f(b) &= f(p(a, b, c) \cdot b) \\
&= f(a \cdot c) \\
&= f(a) \cdot f(c),
\end{align*}
so by uniqueness of the solution to the equation $x \cdot f(b) = f(a) \cdot f(c)$, we conclude:
\[
p(f(a), f(b), f(c)) = f(p(a, b, c)).
\]
Hence, $p$ is a natural transformation.

It remains to show that $p$ is a morphism of magmas. Using the medial law, we compute:
\begin{align*}
(p(a, b, c) \cdot p(a', b', c')) \cdot (b \cdot b') 
&= (p(a, b, c) \cdot b) \cdot (p(a', b', c') \cdot b') \\
&= (a \cdot c) \cdot (a' \cdot c') \\
&= (a \cdot a') \cdot (c \cdot c').
\end{align*}
Thus, by uniqueness of the solution to the equation $x \cdot (b \cdot b') = (a \cdot a') \cdot (c \cdot c')$, we conclude:
\[
p(a, b, c) \cdot p(a', b', c') = p(a \cdot a', b \cdot b', c \cdot c'),
\]
which shows that $p$ is indeed a morphism.
\end{proof}

\section{Weakly Mal'tsev objects}

Weakly Mal’tsev categories \cite{MF3} are characterized by minimal structural requirements: the existence of pullbacks of split epimorphisms along split epimorphisms, together with the condition that, in any such pullback square, the canonical cospan into the pullback is jointly epimorphic. These axioms define a broad and flexible class of categories that properly contains all Mal’tsev categories, while also encompassing examples that lie strictly beyond the Mal’tsev framework—such as commutative magmas with cancellation, distributive lattices, preordered groups, or the dual of the category of topological spaces.

The notion of a weakly Mal'tsev category motivates the following definition of \emph{weakly Mal'tsev object}, which can be formulated in any category with local products. It will be clear that a category is weakly Mal'tsev if and only if every one of its objects is a weakly Mal'tsev object.

\begin{definition}
Let $\C$ be a category with local products. An object $D$ in $\C$ is called a \emph{weakly Mal'tsev object} if, for every kite diagram over $D$,
\begin{equation}\label{kite}
\vcenter{\xymatrix@!0@=4em{
A \ar@<.5ex>[r]^-{f} \ar[rd]_-{\alpha} & B
\ar@<.5ex>[l]^-{r}
\ar@<-.5ex>[r]_-{s}
\ar[d]^-{\beta} & C \ar@<-.5ex>[l]_-{g} \ar[ld]^-{\gamma} \\
& D
}}
\begin{array}{c}
fr=1_{B}=gs,\\
\alpha r=\beta=\gamma s,
\end{array}
\end{equation}
there exists \textbf{at most one morphism} $\varphi\colon A\times_B C \to D$ such that $\varphi e_1 = \alpha$ and $\varphi e_2 = \gamma$. When such a morphism $\varphi$ exists, the diagram is said to be \emph{admissible}, and $\varphi$ is called the \emph{admissibility morphism} of the kite.
\end{definition}

For any category \( \C \) with local products, we denote by \( \C^{*} \) the full subcategory of \( \C \) consisting of its weakly Mal'tsev objects.

This leads to the following problem: given a category \( \C \) with local products, determine its subcategory \( \C^{*} \). We illustrate this with a few examples:

\begin{enumerate}
\item \( \mathbf{Set}^{*} = \mathbf{2} \),
\item \( \mathbf{Lat}^{*} = \mathbf{DLat} \),
\item \( \mathbf{CSmg}^{*} = \mathbf{CCSmg} \),
\end{enumerate}
where \( \mathbf{2} \) denotes the category generated by a single arrow \( (0 \to 1) \). The result that a lattice is a weakly Mal'tsev object if and only if it is distributive appears in \cite{MF9}, while the characterization of commutative cancellative semigroups as the weakly Mal'tsev objects in the category of commutative semigroups is established in the preprint \cite{MF9a}.

The following example  shows that associativity plays no role in the weakly Mal'tsev property, since the same result that holds for commutative semigroups extends to commutative magmas.

\begin{proposition} For a  commutative magma $(D,\cdot)$, in the category of  commutative magmas, the following conditions are equivalent:
\begin{enumerate}
\item it is a weakly Mal'tsev object;
\item it admits cancellation, i.e., for every $x,y,b\in D$, if $x\cdot b=y\cdot b$ then $x=y$;
\item given any $a,b,c\in D$, the equation $x\cdot b=a\cdot c$ has at most one solution $x\in D$. 
\end{enumerate}

\end{proposition}

This means that $\mathbf{CMag}^{*}=\mathbf{CCMag}$ thus generalizing the case of commutative semigroups. 

We observe further that the result is independent of the specific interaction between two or more magma  operations, provided they cooperate to ensure a form of cancellation. This is illustrated in the following example, which generalizes the notion of distributive lattices. By a \emph{dimagma}, we simply mean a set equipped with two binary operations (magmas), and we refer to it as a \emph{commutative dimagma} when both operations are commutative.

\begin{proposition} For a  commutative dimagma $D=(D,\cdot,+)$, in the category of  commutative dimagmas, the following conditions are equivalent:
\begin{enumerate}
\item it is a weakly Mal'tsev object;
\item it admits joint cancellation, i.e., for every $x,y,b\in D$, if $x\cdot b=y\cdot b$ and $x+ b=y+ b$ then $x=y$;
\item given any $a,b,c\in D$, the system of equations
\begin{align*}
\begin{array}{c}
x\cdot b=a\cdot c\\
x+ b=a+ c
\end{array}
\end{align*}
 has at most one solution $x\in D$. 
\end{enumerate}
\end{proposition}

Clearly, this observation can be extended to any number of operations, whether finite or infinite. In particular, the case of lattices fits into this framework: it is well known that a lattice is distributive if and only if it admits joint cancellation. Thus, distributive lattices arise as a special case where the two operations---meet and join---cooperate to ensure cancellation.

As expected, the commutativity of the operations is also not essential, as illustrated by the following result. Here, by a \emph{unary monoid} we mean an algebraic system \( (D, \cdot, 1, \bar{(\,)}) \), where \( (D, \cdot, 1) \) is a monoid (not necessarily commutative), and \( \bar{()} \) is a unary operation such that for all \( x, y \in D \), the identity
\[
x\bar{y}y = y\bar{y}x
\]
holds.

\begin{proposition} For a  unary monoid $D=(D,\cdot,1,\bar{()})$, in the category of unary monoids, the following conditions are equivalent:
\begin{enumerate}
\item it is a weakly Mal'tsev object;
\item given any $a,b,c\in D$, the equation $x \bar{b}b=a\bar{b} c$ has at most one solution $x\in D$. 
\end{enumerate}
\end{proposition}


It is clear that this property remains unchanged even when the algebras are equipped with additional structure, such as an order or a topology. This contrasts with the classical Mal'tsev context, where, for example, preordered groups are weakly Mal'tsev but not Mal'tsev \cite{MF37} (see also the introduction of \cite{MF38}). Further examples on weakly Mal'tsev categories can be found in \cite{MF24}.


\section{The Unification Theorem}


In this section, we refine the main theorem from \cite{MF38}. Instead of requiring that the class of spans \( \M \) is pullback-stable and contains all identity spans—which, as a consequence, includes all spans arising from local products—we now assume only that \( \M \) is closed under the kernel pair construction; that is, if a span \( (D, d, c) \) is in \( \M \), then the span \( (D(d,c), \dom, \cod) \) is also in \( \M \). We continue to require that for every local product
\begin{equation}\label{ another local product}
\vcenter{\xymatrix{
A \ar@<-.5ex>[r]_-{e_1} & A \times_B C \ar@<-.5ex>[l]_-{\pi_1} \ar@<.5ex>[r]^-{\pi_2} & C \ar@<.5ex>[l]^-{e_2}
}}
\end{equation}
the span \( (A \times_B C, \pi_1, \pi_2) \) belongs to \( \M \). The condition asserting the compatibility of a split square and a span is now replaced by the Kite Condition (item (11) in the following theorem).

\begin{theorem}\label{thm Main}
Let \( \C \) be a category with local products, and let \( \M \) be a class of spans in \( \C \) that is closed under the kernel pair construction and contains the span part of every local product. Then the following conditions are equivalent:
\begin{enumerate}
\item The functor $\Grpd(\C,\M)\to\RG(\C,\M)$ is an isomorphism;
\item The functor $\Grpd(\C,\M)\to\RG(\C,\M)$ has a section;
\item The functor $\Cat(\C,\M)\to\RG(\C,\M)$ is an isomorphism;
\item The functor $\Cat(\C,\M)\to\RG(\C,\M)$ has a section;
\item The functor $\UMG(\C,\M)\to\RG(\C,\M)$ is an isomorphism;
\item The functor $\UMG(\C,\M)\to\RG(\C,\M)$ has a section;
\item The functor $\PreGrpd(\C,\M)\to\Span(\C,\M)$ is an isomorphism;
\item The functor $\PreGrpd(\C,\M)\to\Span(\C,\M)$ has a section;
\item The functor $\MKite(\C,\M)\to\Kite(\C,\M)$ is an isomorphism;
\item The functor $\MKite(\C,\M)\to\Kite(\C,\M)$ has a section;

\item For every diagram in \( \C \) of the form
\begin{equation}\label{diag: kite main thm}
\vcenter{\xymatrix@!0@=4em{
& E \ar[dd]^-{\beta}
\ar@<.5ex>[dl]^-{p_1}
\ar@<-.5ex>[dr]_-{p_2}
\\
A \ar@<.5ex>[ru]^-{e_1} \ar[rd]_-{\alpha} 
&  & 
C \ar@<-.5ex>[lu]_-{e_2} \ar[ld]^-{\gamma} \\
& D \ar[dl]_{d} \ar[rd]^{c} \\
D_0 && D_1
}}
\end{equation}
in which:
\begin{itemize}
    \item the pair \( (p_1, p_2) \) is jointly monic;
    \item the spans \( (D, d, c) \) and \( (E, p_1, p_2) \) belong to \( \M \);
    \item the following identities hold:
    \begin{align*}
    & p_1 e_1 = 1_A, \quad p_2 e_2 = 1_C, \\
    & (e_1 p_1)(e_2 p_2) = (e_2 p_2)(e_1 p_1), \\
    & \alpha(p_1 e_2 p_2) = \beta = \gamma(p_2 e_1 p_1), \\
    & d \alpha p_1 = d \alpha(p_1 e_2 p_2), \quad c \gamma p_2 = c \gamma(p_2 e_1 p_1),
    \end{align*}
\end{itemize}
then there exists a unique morphism \( m \colon E \to D \) such that:
\begin{align*}
m e_1 &= \alpha, \\
m e_2 &= \gamma, \\
dm &=d\gamma p_2\\cm &=c\alpha p_1.    \end{align*}

\end{enumerate}

\end{theorem}
\begin{proof}
It is clear that whenever a forgetful functor is an isomorphism, it necessarily admits a section. Since there are forgetful functors from internal groupoids to internal categories, and from internal categories to unital multiplicative graphs, we conclude that \((2) \Rightarrow (4) \Rightarrow (6)\). 

Similarly, we have \( (10) \Rightarrow (8) \Rightarrow (6) \). 

More specifically, the implication \( (10) \Rightarrow (8) \) follows directly from diagram~(\ref{diag: kite5}). To show \( (8) \Rightarrow (6) \), proceed as follows: take any reflexive graph whose span part belongs to \( \M \), and equip it with its canonical pregroupoid structure, say \( p(x, y, z) \). Then define the multiplication on the reflexive graph by setting \( m(x, y) = p(x, 1, y) \). The naturality of \( m \) follows directly from the naturality of \( p \).

Our main task is to show that \((6) \Rightarrow (11)\). Once this implication is established, we proceed to show that \((11)\) implies all other conditions in the theorem.


We proceed in a similar manner as in the proof of Lemma~\ref{lemma 1}.

Consider any kite diagram as in (\ref{diag: kite main thm}) satisfying the prescribed hypotheses. Since \( D(d,c) \) is in \( \M \), and \( \M \) is closed under the kernel pair construction, we conclude that the reflexive graph \[ (D(c,d), D, \dom, \Delta, \cod) \] is an object in \( \RG(\C, \M) \). 
By applying the section to the forgetful functor \( \UMG(\C,\M) \to \RG(\C,\M) \), we obtain the canonical multiplication morphism
\[
\mu \colon D(c,d) \times_D D(c,d) \to D(c,d).
\]
Moreover, we define a morphism
\[
\theta \colon E \to D(c,d) \times_D D(c,d)
\]
as
\[
\theta = \langle\langle \alpha p_1, \alpha p_1, \beta \rangle, \langle \beta, \gamma p_2, \gamma p_2 \rangle\rangle.
\]
The desired morphism \( m \colon E \to D \) is then obtained as \( m = \midop \circ \mu \circ \theta \).

The object \( D(c,d) \), along with the morphisms \( \dom, \midop, \cod \) and the diagonal \( \Delta \colon D \to D(c,d) \), are defined as discussed in the subsection on the kernel pair construction (see Subsection~\ref{subsec: kernel pair construction}). 

It is a matter of routine calculation to verify that the morphism \( \theta \)  is well defined, and that \( m \) satisfies the required conditions. This completes the proof of existence.

In order to prove uniqueness we observe that there exists also a morphism $\delta\colon{E\to E(p_1,p_2)\times_E E(p_1,p_2)}$ defined as \begin{align*}
\delta=\langle
\langle e_1p_1,e_1p_1,e_1p_1e_2p_2
\rangle,\langle e_1p_1e_2p_2,e_2p_2,e_2p_2
\rangle\rangle
\end{align*}
and it satisfies $\midop\circ\mu\circ
\delta=1_E$. Indeed, since $(p_1,p_2)$ is jointly monic and we have
\begin{align*}
p_1\circ\midop \circ\mu\circ\delta &=p_1\circ\cod\circ
\mu\circ\delta\\
&=p_1\circ\cod\circ
\pi_1\circ\delta\\
&=p_1\circ\cod\circ\langle e_1p_1,e_1p_1,e_1p_1e_2p_2
\rangle\\
&=p_1e_1p_1=p_1
\end{align*}
and
\begin{align*}
p_2\circ\midop \circ\mu\circ\delta &=p_2\circ\dom\circ
\mu\circ\delta\\
&=p_2\circ\dom\circ
\pi_2\circ\delta\\
&=p_2\circ\dom\circ\langle e_1p_1e_2p_2,e_2p_2,e_2p_2
\rangle\\
&=p_2e_2p_2=p_2
\end{align*}
we may conclude that $\midop\circ\mu\circ
\delta=1_E$.

Given that the span $(E,p_1,p_2)$ is in $\M$, and hence $$(E(,p_1,p_2),E,\dom,\Delta,\cod)$$ is in $\RG(\C,\M)$, by the  naturality of $\mu$, as displayed,
\begin{equation}\label{diag: mu is natural again}
\vcenter{\xymatrix{
E(p_1,p_2)\times_E E(p_1,p_2) \ar[d]_{m^3\times_m m^3}\ar@<0ex>[r]^-{\mu} & E(p_1,p_2) \ar@<0ex>[r]^-{\midop} \ar@<0ex>[d]^-{m^3} & E \ar@<0ex>[d]^-{m}\\
D(c,d)\times_D D(c,d) \ar[r]^-{\mu} & D(c,d) \ar[r]^-{\midop} & D
}}
\end{equation}
 and observing further that $\theta=(m^3\times_m m^3)\delta$ we obtain

 \begin{align*}
  m &= m1_E=m\circ\midop\circ\mu
  \circ\delta\\
  &=\midop\circ\mu\circ (m^3\times_m m^3)\circ \delta\\
  &=\midop\circ\mu\circ
  \theta. 
 \end{align*}
 This means that $m$ is uniquely determined as $\midop\circ\mu\circ\theta 
$.


We have thus established the implication \( (6) \Rightarrow (11) \).

Our final task is to show that condition \( (11) \) implies condition \( (9) \), and that all remaining conditions follow from there.

\begin{enumerate}
\item[] [\( (11) \Rightarrow (9) \)]  
Given a directed kite whose direction span lies in \( \M \), we may complete it with a local product as illustrated in diagram~(\ref{diag: kite with local product}). This yields a kite diagram satisfying the hypotheses of condition \( (11) \) in our theorem. It is clear that all the required assumptions are fulfilled, thereby guaranteeing the existence of the desired multiplicative structure on the directed kite. Naturality follows from the ability to construct a suitable kite diagram analogous to the one shown in diagram~(\ref{diag: kite7}).

\item[] [\( (9) \Rightarrow (7) \)]  Existence and uniqueness follows from  diagram $(\ref{diag: kite5})$ whereas naturality follows by a diagram similar to (\ref{diag: kite7}). 

\item[] [\( (9) \Rightarrow (5) \)]  Existence and uniqueness follows from  diagram $(\ref{diag: kite1})$ whereas naturality follows from diagram  (\ref{diag: kite4}). 

\item[] [\( (9) \Rightarrow (3) \)]  Existence and uniqueness follows from  diagram $(\ref{diag: kite2})$ whereas naturality follows from diagram  (\ref{diag: kite4}). 

\item[] [\( (9) \Rightarrow (1) \)]  Existence and uniqueness follows from  diagram $(\ref{diag: kite3})$ whereas naturality follows from diagram  (\ref{diag: kite4}). 

\end{enumerate}
Further details on the above observations can be found in \cite{MF7,MF17,MF38}.
\end{proof}

In the following section, we analyze the three well-known classes of spans that give rise to the three distinct types of the Mal'tsev condition.

\section{Naturally Mal'tsev, Mal'tsev, and Weakly Mal'tsev}

In order to articulate the three notions of Mal'tsev-type categories under consideration, we begin by introducing the corresponding classes of spans that characterize them. We also recall the standard definitions and introduce some auxiliary notation related to local products. These classes serve as parameters in the formulation of the main characterizations and provide a framework for a uniform treatment across varying levels of generality, suggesting that additional classes may emerge in future developments.

Let us consider, in a category $\C$ with local products, the following three classes of spans:
\begin{enumerate}
\item[$\M_0$:] The class of all spans $(D, d, c)$ for which the kernel pairs of $d$ and $c$ exist.
\item[$\M_1$:] The class of all spans $(D, d, c)$ for which the kernel pairs of $d$ and $c$ exist, and the pair $(d, c)$ is jointly monic.
\item[$\M_2$:] The class of all spans $(D, d, c)$ for which the kernel pairs of $d$ and $c$ exist, and the pair $(d, c)$ is jointly strongly monic.
\end{enumerate}

We adopt the following standard definitions:
\begin{enumerate}
\item $\C$ is said to be of \emph{naturally Mal'tsev type} if the Lawvere Condition holds; that is, every reflexive graph has a unique groupoid structure.
\item $\C$ is said to be of \emph{Mal'tsev type} if every reflexive relation is a tolerance relation; that is, reflexive and symmetric.
\item $\C$ is said to be of \emph{weakly Mal'tsev type} if, in every local product diagram
\begin{equation}\label{local product --- epimorphic}
\vcenter{\xymatrix{
A \ar@<-.5ex>[r]_-{e_1} & E \ar@<-.5ex>[l]_-{p_1} \ar@<.5ex>[r]^-{p_2} & C \ar@<.5ex>[l]^-{e_2}
}}
\end{equation}
the pair $(e_1, e_2)$ is jointly epimorphic.
\end{enumerate}

As a shorthand, we will say that a local product is \emph{epimorphic} when the pair \( (e_1, e_2) \) is jointly epimorphic. Similarly, we will say that a local product is a \emph{local coproduct} when, in the induced split square displayed in diagram~\eqref{split square}, the pair \( (e_1, e_2) \) forms the pushout of the span \( (B, r, s) \). In other words, the object \( E \) simultaneously realizes both the pullback \( A \times_B C \) and the pushout \( A +_B C \).

Finally, we will say that a local product, such as the one displayed in~\eqref{local product --- epimorphic}, is \emph{extremal} if for every span \( (D, d, c) \) in \( \M_1 \), and for every pair of morphisms \( \alpha \colon A \to D \), \( \gamma \colon C \to D \), satisfying the compatibility conditions \( d \alpha = d \gamma p_2 e_1 \) and \( c \gamma = c \alpha p_1 e_2 \), there exists a unique morphism \( m \colon E \to D \) such that
\[
dm = d \gamma p_2, \qquad cm = c \alpha p_1, \qquad m e_1 = \alpha, \qquad m e_2 = \gamma.
\]

Note that, in the presence of limits and colimits, such a local product is extremal if and only if for every commutative square
\[
\xymatrix{
A + C \ar[r]^-{[e_1, e_2]} \ar[d]_{[\alpha, \gamma]} & E \ar@{-->}[ld] \ar[d]^{\langle d \gamma p_2,\, c \alpha p_1 \rangle} \\
D \ar[r]_-{\langle d, c \rangle} & D_0 \times D_1
}
\]
in which \( \langle d, c \rangle \) is a monomorphism, there exists a unique diagonal morphism \( m \colon E \to D \) making the diagram commute. It follows that, in general, this notion lies strictly between the pair \( (e_1, e_2) \) being an extremal epimorphism and being a strong epimorphism.

\begin{theorem}\label{them: mainII}
Let $\C$ be a category with local products.
\begin{enumerate}
\item The category $\C$ is of naturally Mal'tsev type if and only if the equivalent conditions of Theorem~\ref{thm Main} hold for $\M = \M_0$. Moreover, assuming the existence of binary products and a terminal object, each object in $\C$ admits a canonical Mal'tsev operation. Furthermore, if pushouts of split monomorphisms along split monomorphisms exist, then every local product is also a local coproduct.

\item The category \( \C \) is of Mal'tsev type if and only if the equivalent conditions of Theorem~\ref{thm Main} hold for \( \M = \M_1 \). Moreover, these conditions are further equivalent to the requirement that every local product in \( \C \) is extremal.

\item Assuming the existence of equalizers, as well as the existence of kernel pairs for composites of a split epimorphism with an equalizer, the category~$\C$ is of weakly Mal'tsev type if and only if the equivalent conditions of Theorem~\ref{thm Main} hold for $\M = \M_2$.

\end{enumerate}
\end{theorem}
\begin{proof}
Each of the classes $\M_0$, $\M_1$, and $\M_2$ is closed under the kernel pair construction and contains every span arising as a local product. Indeed, if $(D, d, c)$ is a span for which the kernel pairs of $d$ and $c$ exist, then the span $(D(d,c), \dom, \cod)$ obtained via the kernel pair construction also belongs to the same class, and both $\dom$ and $\cod$ admit kernel pairs, as they are split epimorphisms.

The case of $\M_0$ is immediate. For $\M_1$, it is well known that a span $(D, d, c)$ is a relation if and only if the induced span $(D(d,c), \dom, \cod)$ is a relation. In the case of $\M_2$, it suffices to note that jointly strongly monomorphic pairs are stable under pullback.

Let us now analyze each of the three cases:
\begin{enumerate}
\item The result follows from the fact that the first condition in Theorem~\ref{thm Main} is precisely the Lawvere Condition. The existence of a canonical Mal'tsev operation has been established in Theorem~\ref{thm Naturally}. 

It remains to show that the condition that every local product is a local coproduct can be added to the list. Assuming that pushouts of split monomorphisms along split monomorphisms exist, the Kite Condition implies that the canonical comparison morphism \( A +_B C \to A \times_B C \) is an isomorphism.

Conversely, given any span \( (D,d,c) \) in \( \M_0 \), we take its kernel pair construction and consider the associated kite diagram as illustrated in~(\ref{diag: kite5}). By hypothesis, this diagram can be completed by a local product in which the pair \( (e_1, e_2) \) forms the pushout of \( \Delta \) along itself. This yields the desired pregroupoid structure on the span \( (D,d,c) \).

\item Let us first show that the condition that every reflexive relation is a tolerance relation is equivalent to the condition that every relation is difunctional. This confirms that the Mal'tsev condition in our definition is indeed equivalent to the other conditions in Theorem~\ref{thm Main} for \( \M = \M_1 \).

Indeed, suppose that every reflexive relation is symmetric. Then, for any given relation \( (D, d, c) \), the kernel pair construction \( K(D, d, c) \) produces a reflexive relation whose symmetry equips \( (D, d, c) \) with a (unique) pregroupoid structure. This, in turn, implies that \( (D, d, c) \) is difunctional. 

It is important to note, however, that this argument relies on the assumption that the pair \( (d, c) \) is jointly monic. The construction does not, in general, yield a pregroupoid structure for arbitrary spans.

Conversely, assume that every relation is difunctional. It is well known that, under this assumption, every reflexive relation is necessarily symmetric. If \( p(x, y, z) \) expresses the difunctionality condition (as encoded by a pregroupoid structure), then the instance \( p(1, x, 1) \) yields the required symmetry.

It remains to show that the condition stating that every local product is extremal can be added to the list. Clearly, the Kite Condition relative to $\M_1$ implies that every local product is extremal, in the sense defined above. Conversely, given a span $(D,d,c)$ in $\M_1$, consider the kernel pair construction and take $E = A \times_B C$, where $A = D(d)$, $C = D(c)$, $B = D$, $\alpha = d_1$, and $\gamma = c_2$. The assumption that every local product is extremal ensures the existence of a unique diagonal morphism $m \colon E \to D$, which equips the span $(D,d,c)$ with a pregroupoid structure. This shows that $(D,d,c)$ is difunctional, as required.

 \item Recall that a span \( (D, d, c) \) is \emph{jointly strongly monic} if, in the presence of binary products, the morphism \( \langle d, c \rangle \colon D \to D_0 \times D_1 \) is a strong monomorphism—that is, it is orthogonal to the class of all epimorphisms. More specifically, if in the presence of binary products, for any commutative square
\[
\xymatrix{
X \ar[r]^e \ar[d]_f & Y \ar[d]^g \\
D \ar[r]_-{\langle d, c \rangle} & D_0 \times D_1
}
\]
with \( e \) an epimorphism and \( \langle d, c \rangle \circ f = g \circ e \), there exists a unique morphism \( m \colon Y \to D \) such that \( m \circ e = f \) and \( \langle d, c \rangle \circ m = g \). It is well known that this class is pullback stable and hence is closed under the kernel pair construction (for those spans admitting kernel pairs) and contains all spans that are obtained by local products.

If the pair \( (e_1, e_2) \) is jointly epimorphic, then for any diagram as in the Kite Condition, with the span part \( (D, d, c) \) lying in \( \M_2 \), one may construct an orthogonality diagram between the pair \( (e_1, e_2) \) and the pair \( (d, c) \). In this setting, the uniqueness condition for the diagonal morphism corresponds precisely to the existence of a unique morphism \( m \colon E \to D \) that satisfies the Kite Condition.
Conversely, assume the Kite Condition holds, and consider the equalizer of a pair of morphisms \( u, v \colon E \to Z \) such that \( u e_1 = u e_2 \) and \( v e_1 = v e_2 \). Let \( k = \mathrm{eq}(u, v) \colon D \to E \) be the equalizing morphism. Then the span \( (D, p_1 k, p_2 k) \) lies in \( \M_2 \), assuming that the composition of a split epimorphism with an equalizer admits a kernel pair. The Kite Condition then ensures the existence of a unique morphism \( m \colon E \to D \) such that \( m \circ k = \mathrm{id}_D \) and \( k \circ m = \mathrm{id}_E \), thereby showing that \( k \) is an isomorphism. This proves that the pair \( (e_1, e_2) \) is jointly epimorphic.
\end{enumerate}
This concludes the proof.
\end{proof}


\section{Conclusion}



It is well established that the appropriate categorical generalization of a difunctional relation is a pregroupoid \cite{K1,K2}. Similarly, the natural generalization of a preorder—understood as a reflexive and transitive relation—is a unital multiplicative graph \cite{JP}. However, the correct categorical analogue of a tolerance relation in the sense of Poincaré—that is, a reflexive and symmetric relation—remains an open question. In parallel, it would be desirable to develop a notion of permutability of congruences within a purely categorical framework, independent of assumptions specific to regular categories.


In this paper, we have revisited the notions of Mal'tsev, naturally Mal'tsev, and weakly Mal'tsev categories from a unified perspective based on classes of spans and the concept of local products. While these notions differ in generality and prominence, our approach highlights their structural similarities. By relaxing standard assumptions—such as the existence of all finite limits—we have shown that key results continue to hold in more general categorical settings.

The introduction of multiplicative directed kites and weakly Mal'tsev objects offers new tools for analyzing internal categorical structures, extending beyond the classical regular or finitely complete contexts. These generalizations not only broaden the applicability of Mal'tsev-type conditions but also reveal new categorical phenomena that merit further exploration.



\section*{Acknowledgements}
This work has previously been funded by FCT/MCTES (PIDDAC) through the projects: Associate Laboratory ARISE LA-P-0112-2020; UIDP-04044-2020; UIDB-04044-2020; PAMI–ROTEIRO-0328-2013 (02\-2158); MATIS (CENTRO-01-0145-FEDER-000014 - 3362); CENTRO-01-0247-FEDER-(069665, 039969); as well as POCI-01-0247-FEDER-(069603, 039958, 039863, 024533); by CDRSP and ESTG from the Polytechnic of Leiria.

Furthermore, the author acknowledges the project Fruit.PV and Fundação para a Ciência e a Tecnologia (FCT) for providing financial support through the CDRSP Base Funding project (DOI: 10.54499/UI\-DB/04044/2020).

\end{document}